\documentclass{amsart}
\usepackage{amsfonts}
\usepackage{amssymb}
\usepackage{amsmath}
\usepackage{amsrefs}
\usepackage{mathrsfs}

\usepackage{color}
\usepackage[dvipsnames]{xcolor}
\usepackage{tikz}
\usepackage{setspace}
\usepackage{multicol}
\usetikzlibrary{calc,intersections,decorations.pathreplacing}
\usepackage{ragged2e}
\usepackage[linktocpage=true]{hyperref}
\usepackage{pgfplots}
\usepackage{ dsfont }

\usetikzlibrary{arrows.meta}
\usetikzlibrary{arrows}

\newtheorem{thm}{Theorem}[section]
\newtheorem{cor}[thm]{Corollary}
\newtheorem{thmex}[thm]{Theorem--Example}
\newtheorem{lem}[thm]{Lemma}
\newtheorem{prop}[thm]{Proposition}

\newtheorem{conj}[thm]{Conjecture}

\theoremstyle{definition}

\newtheorem{defn}[thm]{Definition}
\newtheorem{prop-def}[thm]{Proposition--Definition}

\theoremstyle{remark}
\newtheorem{rmk}[thm]{Remark}

\newcommand{\zz}{{\mathbb Z}}
\newcommand{\qqq}{{\mathbb Q}}

\newcommand{\ff}{{\mathbb F}}
\newcommand{\kk}{{\mathbb K}}
\newcommand{\sh}{\mathcal}

\usetikzlibrary{fit}

\usetikzlibrary{fit}

\hypersetup{ colorlinks=true, citecolor=cyan, linkcolor=blue, urlcolor=blue}
\begin{document}

\title[Elliptic surfaces and linear systems with fat points]{Elliptic surfaces and linear systems with fat points}

\author[A. Zahariuc]{A. Zahariuc}
\address{Department of Mathematics, UC Davis, One Shields Ave, Davis, CA 95616}
\email{azahariuc@math.ucdavis.edu}


\begin{abstract}
We investigate the expected dimensionality of linear systems with general fat points on certain surfaces using an approach by specialization to elliptic surfaces. For the projectivization of the Atiyah bundle over an elliptic curve with a certain polarization, we observe that the special case of only one fat point implies the general case of arbitrarily many fat points, as well as results concerning other surfaces. We conjecture that this special case holds in characteristic $0$, but prove that it fails in any positive characteristic. 
\end{abstract}

\maketitle

\tableofcontents

\section*{Introduction}
Let $S$ be a nonsingular projective surface over an arbitrary algebraically closed field $\kk$ and ${\sh L}$ a line bundle over $S$ such that $\mathrm{H}^i(S,{\sh L}) = 0$ for $i=1,2$. 

Recall that for $p_1,p_2,...,p_n$ \emph{general} points on $S$ and $m_1, m_2, ..., m_n$ positive integers, the naive ``expected dimension'' of the linear system of divisors in $|{\sh L}|$ with multiplicity at least $m_i$ at $p_i$ is, by definition,
\begin{equation}\label{expdim}
\text{expdim }|{\sh L}(m_1,m_2,...,m_n)| = \max \left\{-1, \dim |{\sh L}| - \sum_{i=1}^n \frac{m_i(m_i +1)}{2} \right\}.
\end{equation}
We are following the convention that empty linear systems have dimension $-1$. 

The expected dimension is a lower bound for the actual dimension. Linear systems for which the two are equal are called \emph{nonspecial}, while the others are called \emph{special}. In the much-studied case $S = {\mathbb P}^2$, we have the famous Segre-Harbourne-Gimigliano-Hirschowitz conjecture, which roughly states that such a linear system is special if and only if, after blowing up the $n$ general (in fact, \emph{very} general if $\deg {\sh L}$ is not fixed beforehand) points $p_1,p_2,...,p_n$, it contains a multiple $(-1)$-curve in its base locus.

The main example we will be concerned with is the surface obtained by projectivizing the unique rank $2$ vector bundle ${\sh V}$ on a genus one curve $E$, which can be obtained from a nontrivial extension of the form $0 \to {\sh O}_E \to {\sh V} \to {\sh O}_E \to 0$. Since ${\sh V}$ is often called the ``Atiyah vector bundle,'' we will agree to call its projectivization $\smash{ S = \underline{\mathrm{Proj}}_E\mathrm{Sym }({\sh V}^\vee) }$ the Atiyah ruled surface. However, we will only be concerned with one specific line bundle, namely ${\sh L} = {\sh O}_S(F_q+\ell E_\infty)$, where $F_q$ is the fiber of some arbitrary point $q \in E(\kk)$, $E_\infty$ is the section of $S$ induced by the first copy of the structure sheaf in the extension above and $\ell$ is a positive integer. 

The main observation concerning the Atiyah ruled surface is that any collection of $n \geq 1$ general fat points $\mathrm{Spec}({\sh O}_{S,p_i}/\mathfrak{m}_{p_i}^{m_i}) \subset S$ of arbitrary weights $m_i$ impose the expected number of conditions on $|F_q+\ell E_\infty|$ (for all $\ell \geq 0$), if and only if this condition holds for $n=1$, cf. part (a) of Theorem \ref{main theorem 2}. 

The reduction relies on a simple degeneration argument and it applies equally to projective surfaces which admit deformations to elliptic surfaces over the projective line, with a line bundle whose associated complete linear system consists only of divisors which are a rigid section of the fibration plus several fibers. The general form of our result, Proposition \ref{main theorem}, states that the case $n=1$ of the problem above for the Atiyah ruled surface implies the statement that such general surfaces have the property that any number of general fat points of arbitrary weight impose the expected number of conditions on the respective divisor class. We give some concrete examples in Theorem \ref{main theorem 2}, perhaps the most interesting being that of K3 surfaces. However, the special case of one fat point on the Atiyah surface has turned out to be surprisingly subtle and he expected dimensionality of the suitable linear system in this case is false in positive characteristic and currently open in characteristic zero.

The case of K3 surfaces with multiples of the polarizing class and homogeneous linear systems has been previously studied by De Volder and Laface ~\cite{[DL05]}. Their main result is that the analogue of the SHGH problem holds for homogeneous linear series on a general genus $g$ K3 surface with $4^u9^v$ fat points, if and only if it holds for just one fat point. Thus the conjectural reduction to the case $n=1$ appears to be somewhat of a theme for this circle of questions.

Although the case of the Atiyah surface is proposed as a conjecture in characteristic zero, it remarkably fails in any positive characteristic for fairly obvious reasons. This boils down, as usual, to the presence of multiple curves in the base locus; however, in contrast to the case of the SHGH Conjecture, these curves are \emph{elliptic} and only exist in positive characteristics. This allows us to exhibit examples of linear systems with general fat points of different dimensions in characteristic $0$ and characteristic $p$.

\begin{thmex}\label{example theorem}
Let $E$ be a smooth genus one curve defined over an algebraically closed field $\kk$ and $q \in E(\kk)$. Let $S$ be the Atiyah ruled surface associated to $E$, $E_\infty \subset S$ the distinguished section and $F_q$ the fiber through $q$ of the ruling.

Let $\ell,m_1,m_2,...,m_n$ be positive integers and $p$ a prime number such that 
\begin{equation}\label{0.2cond} 
2p(m_1+m_2+...+m_n) + n(p^2-p) \leq 2 \ell < m_1^2 + m_2^2 + ... + m_n^2
\end{equation}
and $m_1,m_2,...,m_n \geq p+1$. Consider the (nonspecial) class ${\sh L}_\ell = {\sh O}_S(F_q + \ell E_\infty)$.

Then $|{\sh L}_\ell(m_1,m_2,...,m_n)|$ is empty if $\mathrm{char}(\kk) = 0$, but nonempty if $\mathrm{char}(\kk) = p$.
\end{thmex}

Finally, the elliptic curves which are responsible for speciality in positive characteristic offer an interesting approach for proving non-speciality in characterisic zero by lifting from a suitably chosen positive characteristic. Using this idea, we will prove that the statement for the Atiyah ruled surfaces and a single fat point is indeed true in charactertic $0$ if the multiplicity of the point doesn't exceed $3$. Although we have also found a messy elementary proof of this low-degree case, the proof going through characteristic $3$ is more elegant and shows some potential for generalization. Thus for each part of Theorem \ref{main theorem 2}, we obtain a corresponding unconditional result in characteristic $0$ if we impose the condition that no multiplicity exceeds $3$. For instance, in the case of K3 surfaces the result is the following.

\begin{thm}\label{clear main theorem}
Let $(S,{\sh L})$ be a general projective K3 surface of degree $2g-2$. If $m_1,m_2,...,m_n \leq 3$, then the linear system $|{\sh L}(m_1,m_2,...,m_n)|$ of curves in $|{\sh L}|$ with $n$ general fat points of multiplicities $m_1,m_2,...,m_n$ has the expected dimension (\ref{expdim}).

Moreover, if Conjecture \ref{main conjecture} is true in characteristic zero, then the statement holds without any restriction on the multiplicities $m_i$.
\end{thm}

\noindent \emph{Acknowledgements.} The content of this note, as well as other related ideas, have been discussed at great length with Brian Osserman. I would like to gratefully acknowledge his insights and suggestions, which are an integral part of the paper. I would also like to thank the referee for many useful comments, which helped fix some inaccuracies and also improved the overall quality and clarity of the paper.

\section{Clebsch-Gordan in characteristic $p$ and Atiyah's classification}

The vector bundles over an elliptic curve have been classified in Atiyah's famous paper ~\cite{[At57]}. In this section, we record for future use a simple remark concerning their behavior in positive characteristic, which is hinted at, but not stated explicitly in ~\cite{[At57]}.

If $V$ is a $\kk$-vector space of dimension $2$ and ${\mathrm{char}}(\kk) = 0$, the "Clebsch-Gordan" isomorphism
\begin{equation}\label{CG case} 
V \otimes \mathrm{Sym}^{n-1} V \cong \mathrm{Sym}^n V \oplus \Lambda^2 V \otimes \mathrm{Sym}^{n-2} V
\end{equation}
is essentially a special case of the formula for tensor products of representations of ${\mathfrak{sl}}_2$. It is a standard fact in representation theory that if ${\mathrm{char}}(\kk) = p > 0$, the canonical isomorphism above still holds for $n < p$, but fails for $n=p$. 

Let $E$ be a nonsingular genus one curve over $\kk$.

Atiyah proved that there exists a unique rank $r$ indecomposable locally free sheaf ${\sh F}_r$ which can be obtained by repeated extensions by ${\sh O}_E$. Alternatively, ${\sh F}_r$ can be described as the unique rank $r$ indecomposable locally free sheaf with trivial determinant and $h^0({\sh F}_r) = 1$. In characteristic $0$, ${\sh F}_{r+1}$ is the $r$th symmetric power of ${\sh F_2}$ ~\cite[Theorem 9]{[At57]}.

\begin{lem}\label{lemma 1.1}
If $r < p = \mathrm{char}(\kk) $, then $\mathrm{Sym}^r {\sh F}_2 \cong {\sh F}_{r+1}$. However, $h^0(\mathrm{Sym}^p{\sh F}_2) \geq 2$, so in particular $\mathrm{Sym}^p{\sh F}_2 \neq {\sh F}_{p+1}$. 
\end{lem}

\begin{proof}
We proceed inductively on $r$. The cases $r=0,1$ are trivial. Assume that $r<p$ and that the statement holds for all $r'<r$. Applying (\ref{CG case}) fiberwise to the vector bundle ${\sh F}_2$, we obtain ${\sh F}_2 \otimes \mathrm{Sym}^{r-1} {\sh F}_2 \cong \mathrm{Sym}^r{\sh F}_2 \oplus \mathrm{Sym}^{r-2}{\sh F}_2$, hence 
\begin{equation}\label{(2)}
{\sh F}_2 \otimes {\sh F}_r \cong \mathrm{Sym}^r{\sh F}_2 \oplus {\sh F}_{r-1}
\end{equation}
by the inductive hypothesis. However, hypothesis $(H_r)$ in ~\cite{[At57]} is satisfied for $r < p$, so the left hand side is just ${\sh F}_{r+1} \oplus {\sh F}_{r-1}$ ~\cite[Lemma 20]{[At57]}. Therefore, $\mathrm{Sym}^r {\sh F}_2 \cong {\sh F}_{r+1}$ since the required cancellation property holds by \cite[Theorems 1 and especially 3]{[At56]}.

Because hypothesis $(H_r)$ fails for $r=p$ ~\cite[remark after Theorem 9]{[At57]}, we have ${\sh F}_2 \otimes {\sh F}_p \cong {\sh F}_p \oplus {\sh F}_p$. Although (\ref{(2)}) no longer holds in this case, we can still say that $\mathrm{Sym}^p{\sh F}_2$ is a quotient vector bundle of 
$$ {\sh F}_2 \otimes \mathrm{Sym}^{p-1}{\sh F}_2 \cong {\sh F}_2 \otimes {\sh F}_p \cong {\sh F}_p \oplus {\sh F}_p. $$
Let ${\sh E}$ be any indecomposable summand of $\mathrm{Sym}^p{\sh F}_2$, so that ${\sh E}$ is also a quotient vector bundle of ${\sh F}_p \oplus {\sh F}_p$. First, we claim that $h^0({\sh E}) \neq 0$. 

Assume by way of contradiction that $h^0({\sh E}) = 0$. Any ${\sh O}_E$-module homomorphism ${\sh F}_r \to {\sh E}$ will then have to vanish on ${\sh F}_1 \cong {\sh O}_E \hookrightarrow {\sh F}_r$ and consequently descend to a map ${\sh F}_{r-1} \to {\sh E}$. Therefore, if $h^0({\sh E}) = 0$, we inductively infer that $\mathrm{Hom}({\sh F}_r,{\sh E}) = 0$. In particular, we would have
$$ \mathrm{Hom}({\sh F}_p \oplus {\sh F}_p,{\sh E}) = 0, $$
contradicting the fact that ${\sh E}$ is a quotient of  ${\sh F}_p \oplus {\sh F}_p$.

Thus all indecomposable summands of $\mathrm{Sym}^p{\sh F}_2$ have nonzero global sections, so, in particular, they all have nonnegative degrees. However, $\mathrm{Sym}^p{\sh F}_2$ has degree $0$, so all indecomposable summands have degree $0$ and admit nonzero global sections. By Atiyah's classification, they must be ${\sh F}_r$, for some $r$. By the same theorem, to conclude the proof of the lemma, it suffices to check that $\mathrm{Sym}^p{\sh F}_2$ is not isomorphic to ${\sh F}_{p+1}$. Because the image of any map from ${\sh F}_p$ to ${\sh F}_{p+1}$ is actually contained inside the canonical copy of ${\sh F}_p$ inside ${\sh F}_{p+1}$, the alleged quotient map ${\sh F}_p \oplus {\sh F}_p \to {\sh F}_{p+1}$ cannot be surjective, contradiction. \end{proof}

\section{The case of the Atiyah ruled surface}

\subsection{Generalities on the Atiyah ruled surface}

In this section, we collect the elementary and perhaps well-known facts concerning the Atiyah ruled surface which we will use later on. Let $E$ be a smooth genus one curve with a fixed closed point $q \in E$. Recall Atiyah's classification of vector bundles over an elliptic curve. Since $\mathrm{Ext}^1({\sh O}_E,{\sh O}_E) = \kk$, all nonsplit extensions
\begin{equation}\label{3.1}
0 \longrightarrow {\sh O}_E \longrightarrow {\sh V} \longrightarrow {\sh O}_E \longrightarrow 0
\end{equation}
give rise to the same vector bundle ${\sh V}$ on $E$. The extension being nontrivial boils down to the coboundary map ${\mathrm H}^0({\sh O}_E) \to {\mathrm H}^1({\sh O}_E)$ being nonzero. It follows that $h^0({\sh V}) = 1$, so ${\sh V}$ contains a unique trivial rank one sub bundle. 

Let $S = \mathrm{Proj}_E\mathrm{Sym }({\sh V}^\vee)$ and denote the section of $S \to E$ induced by ${\sh O}_E \subset {\sh V}$ by $E_\infty$. Note that ${\sh N}_{E_\infty/S}$ is isomorphic to the vertical tangent bundle of $S \to E$ restricted to $E_\infty$. However, this is clearly trivial by the description of the tangent bundle of a Grassmannian, so we've shown that ${\sh O}_S(E_\infty)|_{E_\infty}$ is just the structure sheaf. Moreover, it is easy to check that the tautological line bundle ${\sh O}_{{\mathbb P}{\sh V}}(-1)$ is just ${\sh O}_S(-E_\infty)$ and the canonical divisor is $K_S = -2E_\infty$. Of course, any line bundle on $S$ will be of the form $\rho^*{\sh L}' \otimes {\sh O}_S(nE_\infty)$, for some ${\sh L}'$ on $E$ and integer $n$.

It is a well-known fact that the endomorphism ring of ${\sh V}$ is isomorphic to $\kk[\epsilon]/(\epsilon^2)$, with $\epsilon$ corresponding to the nilpotent endomorphism ${\sh V} \twoheadrightarrow {\sh V}/{\sh O}_E \cong {\sh O}_E \hookrightarrow {\sh V}$. Then the $E$-automorphisms of $S$ form a group isomorphic to that of invertible elements of this commutative ring mod scalars, which is clearly just the group ${\mathbf G}_a = (\kk,+)$. These automorphisms act fiberwise on $S$ keeping $E_\infty$ fixed. Consider the "direct symmetries" $\mathrm{Aut}^+(S)$, i.e. the subgroup of the group of automorphisms of $S$, which naturally fits in the following short exact sequence
\begin{equation}\label{3.6}
1 \longrightarrow {\mathbf G}_a \longrightarrow \mathrm{Aut}^+(S) \longrightarrow \mathrm{Pic}^0(E) \longrightarrow 1.
\end{equation}
Note that $\mathrm{Aut}^+(S)$ acts freely and transitively on the set of closed points of $S \backslash E_\infty$.

\begin{lem}\label{lemma 2.1}
If ${\sh L}$ is an effective line bundle on $S$ such that $\deg {\sh L}|_{E_\infty} = 0$, then ${\sh L}$ is ${\sh O}_S(nE_\infty)$, for some integer $n \geq 0$.
\end{lem}

\begin{proof}
As we said, ${\sh L}$ is of the form $\rho^*{\sh L}' \otimes {\sh O}_S(nE_\infty)$, for some line bundle ${\sh L}'$ on $E$ and integer $n$. Taking intersection with a general fiber of $\rho$ we see that $n \geq 0$. The restriction of ${\sh L}$ to $E_\infty$ is just ${\sh L}'$ seen on $E_\infty$, hence $\deg {\sh L}' = 0$. 

Finally, we need to argue that ${\sh L}'$ is trivial. Assume that $E_\infty$ is contained with multiplicity $m$ in the base locus of ${\sh L}$. Then ${\sh L}(-mE_\infty)$ is still effective and has intersection number zero with $E_\infty$. Therefore, a general section $\sigma$ of ${\sh L}(-mE_\infty)$ won't vanish at all on $E_\infty$, so it will trivialize the restriction of ${\sh L}(-mE_\infty)$ to $E_\infty$. In conclusion, ${\sh L}|_{E_\infty}$ is trivial, so ${\sh L}'$ is also trivial, as desired.
\end{proof}

\begin{prop}\label{proposition 2.2}
In characteristic $0$, the linear system $|nE_\infty|$ is a singleton for all $n$. In characteristic $p$, the dimension of the linear system is $\lfloor n/p \rfloor$. Moreover, all members of $|pE_\infty|$ except $pE_\infty$ itself are smooth curves of genus $1$. 
\end{prop}

\begin{proof}
Obviously, ${\mathrm{H}}^0(S,{\sh O}_S(nE_\infty)) = {\mathrm{H}}^0(E,\rho_*{\sh O}_S(nE_\infty))$. Note that $\rho_*{\sh O}_S(E_\infty) \cong {\sh V}^\vee$ and, more generally, 
\begin{equation}\label{2.2}
\rho_*{\sh O}_S(nE_\infty) \cong \mathrm{Sym}^n{\sh V}^\vee \cong \mathrm{Sym}^n {\sh V} = \mathrm{Sym}^n{\sh F}_2.
\end{equation}
Therefore, $h^0({\sh O}_S(nE_\infty)) = h^0(\mathrm{Sym}^n{\sh F}_2)$.

In characteristic $0$, we have $\mathrm{Sym}^n{\sh F}_2 \cong {\sh F}_{n+1}$, so $h^0({\sh O}_S(nE_\infty)) = h^0({\sh F}_{n+1}) = 1$, as desired. The same applies in characteristic $p$ when $n<p$, by Lemma \ref{lemma 1.1}. However, by the same lemma, $\dim |pE_\infty| \geq 1$. By Lemma \ref{lemma 2.1} and the fact that $|nE_\infty|$ is a singleton for $n<p$, all members of $|pE_\infty|$ except $pE_\infty$ itself are reduced and irreducible, hence they must be smooth curves of genus one. Indeed, the arithmetic genus is $1$ by the genus-degree formula, while the geometric genus cannot be zero because of the presence of the map to $E$.

Let $D \in |nE_\infty|$ for arbitrary $n \geq 1$. Let $D = a_1D_1+...+a_kD_k$ with $D_1,...,D_k$ reduced and irreducible. Since $|pE_\infty|$ sweeps out $S$, there exists some divisor $C \in |pE_\infty|$ which intersects $D_i$. However, $(D \cdot C) =0$ and $(D_j \cdot C) \geq 0$ for $j \neq i$, hence $(C \cdot D_i) \leq 0$, which implies $D_i = C$ if $C$ is irreducible or $D_i = E_\infty$, if $C = pE_\infty$. The assertion on the dimension of $|nE_\infty|$ follows immediately.  
\end{proof}

\begin{prop}\label{proposition 2.3}
Let $q \in E(\kk)$ and $F_q$ the fiber of $\rho$ through $q$. Then $\dim |F_q+\ell E_\infty| = \ell$ and the restriction $|F_q+\ell E_\infty| \to |{\sh O}_{F_r}(\ell)|$ is an isomorphism for all closed points $r \neq q$.
\end{prop}

\begin{proof}
First, by the Riemann--Roch theorem for surfaces, we have
\begin{equation}\label{2.3}
\chi({\sh O}_S(F_q+\ell E_\infty)) = \chi({\sh O}_S) + \frac{(F_q + (\ell+2)E_\infty \cdot F_q + \ell E_\infty)}{2} = \ell + 1.
\end{equation}
Since $h^2({\sh O}_S(F_q+ \ell E_\infty)) = 0$ by Serre duality, it follows that $\dim |F_q+\ell E_\infty| \geq \ell$. We argue inductively on $\ell$ that equality occurs. The base case is trivial. Fix $b \in E_\infty(\kk)$. Any divisor in $|F_q+\ell E_\infty|$ containing $b$ must split off a copy of $E_\infty$. Hence $|F_q+\ell E_\infty|$ contains a subspace of codimension at most one consisting of divisors of the form $E_\infty + D'$, $D' \in |F_q + (\ell-1) E_\infty|$ and the inductive step is complete. Hence $\dim |F_q+\ell E_\infty| = \ell$ and also $h^1({\sh O}_S(F_q+ \ell E_\infty)) = 0$. The restriction $|F_q+\ell E_\infty| \to |{\sh O}_{F_r}(\ell)|$ is an isomorphism by trivial dimension considerations. \end{proof}

\subsection{Statement of the main conjecture and some remarks}

In this section, we state the main open problem and give some evidence supporting it. The problem is seemingly very simple: it simply asks how big can the multiplicity of a curve in $|F_q+\ell E_\infty|$ be at a general point $x$ in $S$. The fact, implicit in the definition below, that all algebraically closed fields of the same characteristic behave in the same way is a standard consequence of the Lefschetz principle.

Note that if $x \in S$ is a closed point not on $E_\infty$, then the dimension of the space of global sections of ${\sh O}_S(F_q+\ell E_\infty) \otimes {\sh I}_{x,S}^m$ only depends on the linear equivalence class of $\smash{ \rho(x)-q \in \mathrm{Div}^0(E) }$. This is due to the presence of the automorphisms discussed in subsection \S\S2.1. We will write ${\sh D} = {\sh O}_E(\rho(x)-q)$.
 
\begin{defn}\label{definition 2.4}
Let $\lambda=\lambda_{\mathrm{char}(\kk)}(m,E,{\sh D})$ be the minimal $\ell \geq 0$, for which there exists a curve $C \in |F_q+\ell E_\infty|$ of multiplicity at least $m$ at $x$. Let
$$ \lambda_{\mathrm{char}(\kk)}(m,E) = \max_{{\sh D} \in \mathrm{Pic}^0(E)} \lambda_{\mathrm{char}(\kk)}(m,E,{\sh D})  $$
and
$$ \lambda^{\mathrm{gen}}_{\mathrm{char}(\kk)}(m) = \max_{E} \lambda_{\mathrm{char}(\kk)}(m,E).  $$
Similarly, for fixed $\ell \geq 1$, let $\smash{ \mu_{\mathrm{char}(\kk)}(\ell,E,{\sh D}) }$, $\smash {\mu_{\mathrm{char}(\kk)}(\ell,E) }$ and $\smash{ \mu^\mathrm{gen}_{\mathrm{char}(\kk)}(\ell) }$ be the maximal multiplicity of a curve $C \in |F_q+\ell E_\infty|$ at a given point, with the analogous logical quantifiers as in the case of $\lambda$ and min instead of max.
\end{defn}

Occasionally, we will ambiguously write $\lambda_{\mathrm{char}(\kk)}(m)$ instead of $\lambda_{\mathrm{char}(\kk)}(m,E)$. In these cases, it is implicitly understood that the dependence on $E$ is unimportant.

\begin{conj}\label{main conjecture}
If $\mathrm{char}(\kk) = 0$, or $m \leq p = \mathrm{char}(\kk)$, then
\begin{equation}\label{3.4}
\lambda^\mathrm{gen}_{\mathrm{char}(\kk)}(m) = {m+1 \choose 2}
\end{equation}
and if $m > p = \mathrm{char}(\kk)$, then
\begin{equation}\label{3.55}
\lambda^\mathrm{gen}_{\mathrm{char}(\kk)}(m) = {p+1 \choose 2} + p(m-p).
\end{equation}
\end{conj}

The value conjectured in (\ref{3.4}) is trivially an upper bound, since $\dim |F_q+\ell E_\infty| = \ell$. In characteristic $0$, where divisors in $|F_q+\ell E_\infty|$ are reduced and irreducible as long as they don't contain $E_\infty$, we also have a trivial lower bound
\begin{equation}\label{3.5}
\lambda_0(m) \geq {m \choose 2} + 1,
\end{equation}
because $p_a(|F_q+\ell E_\infty|) = \ell$ from the genus-degree formula. In fact, there is a stronger lower bound which can be established quite easily.

\begin{prop}\label{proposition 3.6}
For all $m \geq 1$,
$$ \lambda_0(m) \geq \frac{m^2}{2}. $$
\end{prop}

\begin{proof} The curve $C_m \in |F_q+\lambda_0(m) E_\infty|$ of multiplicity at least $m$ at some fixed general $x \in S$ is unique and irreducible, by the assumption that $\lambda_0(m)$ is minimal. Indeed, if $C_m$ wasn't unique, then we would have (at least) a pencil of such curves. Then, the corresponding line in $|F_q + \ell E_\infty|$ would have to intersect the hyperplane of divisors which contain $E_\infty$, so removing the $E_\infty$ component from this divisor gives a curve with multiplicity at least $m$ and smaller $\ell$, which is a contradiction. Moreover, (\ref{3.5}) shows that the multiplicity of $C_m$ at $x$ is exactly $m$. 

Instead of just one, choose two general points $x,x' \in S(\kk)$. Let $C_m,C'_m \in |F_q+\lambda_0(m) E_\infty|$ be the curves of multiplicity $m$ at $x$ and $x'$ respectively. Let $\varphi \in \mathrm{Aut}^+(S)$ be the unique automorphism which takes $x'$ to $x$ and let 
$$ D_m = \varphi(C'_m) \in |F_{q'}+\lambda_0(m) E_\infty|. $$ 
Of course, $q \neq q'$, so $C_m \neq D_m$. Let $\smash{ \tilde{S} }$ be the blowup of $S$ at $x$ and let $\smash{ \tilde{C}_m}$, $\smash{\tilde{D}_m }$ be the proper transforms of $C_m$, $D_m$. Then, 
$$ 0 \leq (\tilde{C}_m \cdot \tilde{D}_m)_{\tilde{S}} = ( F_{q}+\lambda_0(m) E_\infty - m Y \cdot F_{q'}+\lambda_0(m) E_\infty - m Y ) = 2 \lambda_0(m) - m^2, $$
where $Y \subset \smash{ \tilde{S} }$ denotes the exceptional divisor, completing the proof. \end{proof}

The final goal of this section is to illustrate the interaction between the characteristic $0$ case and the positive characteristic case, by proving that $\lambda^{\mathrm{gen}}_0(3) = 6$, as predicted by Conjecture \ref{main conjecture}, using a positive characteristic method. The conjecture in characteristic $p$ is equivalent to
\begin{equation}
\begin{aligned}
\lambda^\mathrm{gen}_p(m)-\lambda^\mathrm{gen}_p(m-1)=
\begin{cases}
m & \text{ if } m \leq p, \\
p & \text{ if } m \geq p.
\end{cases}
\end{aligned}
\end{equation}
By Proposition \ref{proposition 2.2}, the difference on the left never exceeds $p$. The main conceptual evidence in favor of Conjecture \ref{main conjecture} is the following observation.

\begin{prop}\label{proposition 2.7}
For any prime number $p$ and smooth elliptic curve $E$ over an algebraically closed field $\kk$ of characteristic $p$, we have 
\begin{equation}\label{pinduction}
\lambda_p(p,E,{\sh D}) \geq  p + \lambda_p(p-1,E,{\sh D})
\end{equation}
if ${\sh D}$ is not $p$-torsion. In particular,
$$ \lambda_p(p,E) \geq  p + \lambda_p(p-1,E) $$
and
$$ \lambda^\mathrm{gen}_p(p) \geq  p + \lambda^\mathrm{gen}_p(p-1). $$
\end{prop}

\begin{proof} By Proposition \ref{proposition 2.2}, there exists a smooth curve $C \in |pE_\infty|$. Note that ${\sh O}_C(px) \neq {\sh O}_S(F_q)|_C$ by the assumption on ${\sh D}$. Choose $D \in |F_q+ \lambda_p(p) E_\infty|$ of multiplicity at least $p$ at $x$. The section $1 \in \mathrm{H}^0({\sh O}_S(D))$ restricted to $C$ has a zero of order at least $p$ at $x$. However, ${\sh O}_S(D)|_C = {\sh O}_S(F_q)|_C$, so we obtain a section of the latter degree $p$ line bundle with a zero of order $p$ at $x$. Given the choice of $x$, this section must vanish identically, hence $C \subset D$. Let $D = C + D'$. Then $D' \in |F_q+ (\lambda_p(p)-p) E_\infty|$ has multiplicity at least $p-1$ at $x$ and (\ref{pinduction}) follows. \end{proof}

\begin{cor}\label{ugly} $\lambda_0(3,y^2=x^3-x+1) = 6$. In particular, $\lambda_0^\mathrm{gen}(3) = 6$.
\end{cor}

\begin{proof}
The discriminant of $y^2=x^3-x+1$ is $-368 = -2^4 \cdot 23$. Let $E$ be $y^2=x^3-x+1$ regarded as an elliptic curve over $\mathrm{Spec}(R)$, where
$$ R={\zz}_{46} = \left\{ \frac{n}{46^k}: n,k \in \zz \right\}. $$
The group of ${\ff}_3$-points of $E$ is isomorphic to ${\zz}/7$, cf. \cite[Example 1, \S5.3]{[Poo01]}.

The proof amounts to little more than rewriting everything over $\mathrm{Spec}(R)$. We will take the freedom to reuse analogous notation. Note that $\mathrm{Ext}^1({\sh O}_E,{\sh O}_E) \cong R$ and consider the extension
$$ 0 \longrightarrow {\sh O}_E \longrightarrow {\sh V} \longrightarrow {\sh O}_E \longrightarrow 0 $$
corresponding to $1 \in \mathrm{Ext}^1({\sh O}_E,{\sh O}_E)$. Again, define $S = \mathrm{Proj}_E\mathrm{Sym }({\sh V}^\vee)$ and denote the section of $\rho:S \to E$ induced by ${\sh O}_E \subset {\sh V}$ by $E_\infty \subset S$. As before, we have $ \rho_* {\sh O}_S(\ell E_\infty) \cong \mathrm{Sym}^\ell {\sh V}$. 

Let $q,q':\mathrm{Spec}(R) \to E$ be two sections of $E$ corresponding to two of the three integer solutions $(-1,1)$, $(0,1)$, $(1,1)$. Clearly, there exists a section $\xi:\mathrm{Spec}(R) \to S$ of the projection map $\sigma:S \to \mathrm{Spec}(R)$ whose image is disjoint from $E_\infty$ and such that $\rho \circ \xi = q'$.

The coherent sheaf $\smash{ {\sh F} = {\sh O}_S(\ell E_\infty + F_q) \otimes {\sh I}_{\xi(E),S}^m }$ on $S$ restricts on each fiber $S_{\mathfrak p} = \sigma^{-1}({\mathfrak p})$ to a sheaf ${\sh F}_{\mathfrak p}$, which we pull back to $S_{\overline{\kappa}_{\mathfrak p}}=\mathrm{Spec}\overline{\kappa}_{\mathfrak p} \times_{\mathrm{Spec}{\kappa_{\mathfrak p}}} S_{\mathfrak p}$. The resulting sheaf is denoted by ${\sh F}_{\overline{\kappa}_{\mathfrak p}}$. It is not hard to see that 
$$ {\mathrm H}^0(S_{\overline{\kappa}_{\mathfrak p}},{\sh F}_{\overline{\kappa}_{\mathfrak p}}) \cong \overline{\kappa}_{\mathfrak p} \otimes_{\kappa_{\mathfrak p}} {\mathrm H}^0(S_{\mathfrak p},{\sh F}_{\mathfrak p}) $$ 
and hence by the semicontinuity part of the cohomology and base change theorem,
\begin{equation}
\begin{aligned}
\dim_{\overline{\ff}_3} {\mathrm H}^0(S_{\overline{\ff}_3},{\sh F}_{\overline{\ff}_3}) &= \dim_{\ff_3} {\mathrm H}^0(S_{(3)},{\sh F}_{(3)})  \\
& \geq \dim_{\qqq} {\mathrm H}^0(S_{(0)},{\sh F}_{(0)}) = \dim_{\overline{\qqq}} {\mathrm H}^0(S_{\overline{\qqq}},{\sh F}_{\overline{\qqq}}).
\end{aligned}
\end{equation}
Of course, ${\sh F}_{\overline{\kappa}_{\mathfrak p}}$ is the sheaf we would have previously (that is, anywhere outside of this proof) called ${\sh L}_\ell \otimes {\sh I}_{x,S}^m$, so we've shown that
$$ \lambda_0(3,E_{\overline{\qqq}},{\sh D}_{(0)}) \geq \lambda_3(3,E_{\overline{\ff}_3},{\sh D}_{(3)}), $$
where ${\sh D}_{\mathfrak p}$ is the linear equivalence class of $q'({\mathfrak p})-q({\mathfrak p})$.

It is clear that $\smash{ \lambda_0(3,E_{\overline{\qqq}}) \geq \lambda_0(3,E_{\overline{\qqq}},{\sh D}_{(0)}) }$, so it suffices to prove that 
$$ \lambda_3(3,E_{\overline{\ff}_3},{\sh D}_{(3)}) \geq 6. $$
Note that ${\sh D}_{(3)}$ is not $3$-torsion because it is nonzero and $E({\mathbb F}_3) \cong {\mathbb Z}/7$, so by Proposition \ref{proposition 2.7}, it suffices to prove that $\smash{\lambda_3(2,E_{\overline{\ff}_3},{\sh D}_{(3)}) = 3}$. 

We will provide an improvised argument, avoiding getting into too much detail. The curves in $|F_q+2E_\infty|$ are reduced and irreducible, provided they don't contain a copy of $E_\infty$. To each such curve $C$ we associate the branch divisor of the map $C \to E$. Recall that $C$ is a genus two curve due to the remark after formula (\ref{3.5}). Because $|F_q+2E_\infty|$ is parametrized by a projective plane, while $\smash{ \mathrm{Pic}^2(E) }$ is a genus one curve, the linear equivalence class of the branch divisor must be constant. However, when (irreducible) curves in $|F_q+2E_\infty|$ acquire double points, the two ramification points of $C \to E$ come together at the singularity. Hence all the divisors classes ${\sh D}_{(3)}$ for which $\smash{ \lambda_3(2,E_{\overline{\ff}_3},{\sh D}_{(3)}) < 3 }$ differ from each other by $2$-torsion classes. However, since $E({\mathbb F}_3) \cong {\mathbb Z}/7$ for our specific $E$, there can be \emph{at most one} ${\sh D}_{(3)}$ for which $\smash{ \lambda_3(2,E_{\overline{\ff}_3},{\sh D}_{(3)}) < 3 }$. In conclusion, it must be possible to choose $q$ and $q'$ to be two of the integer solutions $(-1,1)$, $(0,1)$, $(1,1)$ such that $\smash{ \lambda_3(2,E_{\overline{\ff}_3},{\sh D}_{(3)}) = 3 }$, completing the proof. 

Note that the argument in the last paragraph also shows that $\lambda_0^\mathrm{gen}(2) = 3$. \end{proof}

\begin{rmk}
Assume that $\mathrm{char}(\kk) = 0$ and that Conjecture \ref{main conjecture} is true. If $\ell = m(m+1)/2-1$, a trivial calculation for the expected degeneracy of
$$ \mathrm{H}^0({\sh O}_S(F_q+\ell E_\infty)) \otimes {\sh O}_S \longrightarrow {\mathrm J}^m{\sh O}_S(F_q+\ell E_\infty) $$
suggests that curves in $|F_q+\ell E_\infty|$ of multiplicity $m$ at $x$ should exist if $x \in E_\infty$ or if $x$ lives in one of $m(m+1)/2$ fibers of $S$. For $m = \ell =2$, these are the fibers over the three points of $E$ which differ from $q$ by a nonzero $2$-torsion amount. However, what these fibers might be in general is completely mysterious (even for $m=3$, where we've just proved that the statement of \ref{main conjecture} is true).
\end{rmk}

\section{Proof of the main results}

\subsection{Deformation to an elliptic surface} In this section, we state and prove two versions of the main theorem, which in particular imply Theorem \ref{clear main theorem}. The proof of the theorem is a short degeneration argument. The idea to approach interpolation problems on surfaces by degenerating the underlying surface is by now standard ~\cite{[CM98], [CM00], [CDM09], [DL05], [Hui13]}. The choice of the degeneration was inspired by \cite{[Ch99]}. Finally, in \S\S3.3, we prove Theorem \ref{example theorem}.

\begin{prop}\label{main theorem}

Let $\pi:X \to B$ be a smooth projective family of surfaces over a smooth affine curve $B$ and let ${\sh L}_X \in {\mathrm{Pic}}(X/B)$ be relatively ample. Let $X_t = \pi^{-1}(t)$ and ${\sh L}_t = {\sh L}_X|_{X_t}$ for any $t \in B$. Fix $b \in B(\kk)$. Assume that 
\begin{equation}\label{hi vanishing}
\mathrm{H}^i(X_b,{\sh L}_b) = 0 \text{ for } i =1,2
\end{equation}
and that the central fiber $X_b$ has an elliptic fibration $f:X_b \to {\mathbb P}^1$ with the following properties:

$\bullet$ For general $s \in {\mathbb P}^1$, the following short exact sequence on $f^{-1}(s)$ 
$$ 0 \longrightarrow {\sh N}_{f^{-1}(s)/X_b} \longrightarrow {\sh N}_{f^{-1}(s)/X} \longrightarrow {\sh O}_{f^{-1}(s)} \otimes T_b B \longrightarrow 0 $$ 
is not split;

$\bullet$ There exists a fixed section $G \subset X_b$ of $f$, such that any divisor $D \in |{\sh L}_b|$ is the sum of $G$ and $\dim |{\sh L}_b|$ (mobile) fibers of $f$.

Let $t \in B$ general and $(S,{\sh L}) = (X_t,{\sh L}_t)$. If $m_1,m_2,...,m_n \geq 1$, then
\begin{equation}\label{(3.1)}
\dim |{\sh L}(m_1,m_2,...,m_n)| \leq \max \left\{ -1, \dim |{\sh L}| - \sum_{i=1}^{n} \lambda(m_i) \right\},
\end{equation}
where $\lambda(m_i)$ is either the $\lambda$ associated to a fiber of $f$ if this fibration is isotrivial, or $\lambda^\mathrm{gen}(m_i)$ if the fibration is not isotrivial, cf. Definition \ref{definition 2.4}.

\end{prop}
 
\begin{proof}  
Let $r_1,r_2,...,r_n \in {\mathbb P}^1$ general points and $E_1,E_2,...,E_n \subset X_b$ the corresponding elliptic fibers. For each $i$, choose a general point $p_i \in E_i$ and a general tangent vector $v_i \in T_{p_i}X$ such that $(\mathrm{d}\pi)_{p_i}(v_i)=\frac{\partial}{\partial t} \neq 0$, where $\frac{\partial}{\partial t}$ stands for a  nonzero tangent vector at $b \in B$. Possibly after an \'{e}tale base change, we can find sections $\xi_i:B \to X$ of $X \to B$ such that $\xi_i(b) = p_i$ and $(\mathrm{d}\xi_i)_b(\frac{\partial}{\partial t})=v_i$. Note that \'{e}tale base changes preserve the property in the first bullet. Indeed, the normal bundle ${\sh N}_{f^{-1}(s)/X}$ only depends on the first order thickening $\pi^{-1}(2b)$ of $X_b$ in $X$ and on $f^{-1}(s) \subset \pi^{-1}(2b)$. However, for any \'{e}tale base change $B' \to B$ and preimage $b'$ of $b$, the map $X' = X \times_B B' \to X$ induces an isomorphism between the first order thickenings of $X_b$ in $X$ and $X'_{b'}$ in $X'$ respectively and the isomorphism trivially maps the respective copies of $f^{-1}(s)$ to each other, so the normal bundle remains isomorphic to the Atiyah bundle. In particular, the exact sequence cannot split.

Note that (\ref{hi vanishing}) implies that $\mathrm{H}^i(X_t,{\sh L}_t) = 0$, $i > 0$, for all $t$ in a Zariski neighborhood of $b$. We may shrink $B$ to this open subset. By the cohomology and base change theorem, $\pi_*{\sh L}$ is locally free and, moreover, $\pi_*{\sh L} \otimes \kappa_t \cong \mathrm{H}^0(X_t,{\sh L}_t)$ for all $t \in B$. Let $R$ be the classical projectivization of $\pi_* {\sh L}$. Then $R$ is a projective bundle with fibers ${\mathbb P}^{\dim |{\sh L}_b|}$ if $|{\sh L}_b| \neq \emptyset$.

Let $\smash{ {\sh E} =  {\sh L}_X \otimes \prod_{i=1}^n {\sh I}^{m_i}_{\xi_i(B),X} }$ and ${\sh F} = \pi_* {\sh E}$. Clearly, ${\sh E}$ is torsion free and hence so is ${\sh F}$, thus ${\sh F}$ is locally free because it is a torsion free sheaf over a smooth curve. Let $B^\circ = B \backslash \{b\}$ and $\pi^\circ: X \backslash X_b \to B^\circ$ the corresponding family over the punctured $B$. By the semicontinuity part of the cohomology and base change theorem, after possibly further shrinking $B$, we may assume that the function $t \mapsto h^0(X_t,{\sh E}_t)$ is constant, possibly with the exception of the origin $b$, where it might jump up (on the other hand, the generic point of $B$ does not require any special treatment). By the cohomology and base change theorem again, ${\sh F} \otimes \kappa_t \cong \mathrm{H}^0(X_t,{\sh E}_t)$ for all $t \neq b$.

Let ${\sh F}^\circ$ be the restriction of ${\sh F}$ to $X \backslash X_b$ and $P^\circ$ the (classical) projectivization of ${\sh F}^\circ$. Since $P^\circ$ is trivializable, there is no issue with choosing sections of $P^\circ$, i.e. flat families of divisors $D_t \in |{\sh L}_t(m_1,m_2,...,m_n)|$ relative to $\xi_1(t),\xi_2(t),...,\xi_n(t)$, for each $t \neq b$. Let $D^\circ$ be the total space of such a family of divisors and let $D_X$ be its Zariski closure in $X$. 

\begin{lem}\label{lemma inside prop}
If $D_b$ is the fiber of $D_X$ over $b$, then $D_b$ contains $E_i$ with multiplicity at least $\lambda(m_i)$, for all $1 \leq i \leq n$.
\end{lem}

\begin{proof}
Let $Y$ be the blowup of $X$ along $E_1 \cup E_2 \cup ... \cup E_n$. The central fiber $Y_b$ has $n+1$ irreducible components: one isomorphic to $X_b$ and $n$ geometrically ruled surfaces $S_1$, $S_2$, ..., $S_n$. The condition in the first bullet implies that each $S_i$ is the Atiyah ruled surface over $E_i$ and, moreover, $E_i^\infty = S_i \cap X_b$ is the corresponding distinguished section of each such surface. A key observation is that $E_i^\infty$ has trivial normal bundle both inside $S_i$ and inside $X_b$. Of course, $\xi_i$ lifts to $Y$ and, by a slight abuse of notation, we will continue to denote the lift by $\xi_i$. By construction, $\xi_i(b) = x_i$ is a \emph{general} point in $S_i$. 

Let $D_Y$ be the closure of $D^\circ$ inside $Y$, $D'_b = D_Y \cap Y_b$ and ${\sh L}_Y = {\sh O}_Y(D_Y)$. Then $D'_b$ has multiplicity at least $m_i$ at the point $x_i \in S_i$. Let ${\sh L}_i = {\sh O}_{S_i}(D_i)$ be the restriction of ${\sh L}_Y$ to $S_i$, where $D_i = D'_b \cap S_i$. Simply by construction, the restriction of ${\sh L}_i$ to $E_i^\infty$ is isomorphic to the restriction of ${\sh L}_Y|_{X_b}$ to $E_i^\infty$, but the latter is just ${\sh O}_{E_i^\infty}(q_i)$ due to the condition in the second bullet and the observation that the normal bundle of $E_i^\infty$ in $X_b$ is trivial, where $q_i = G \cap E_i$ on the original $X_b$. Let $F_{q_i} \subset S_i$ be the fiber of $q_i$. Hence, ${\sh L}_i \cong {\sh O}_{S_i}(F_{q_i} + \ell_i E_i^\infty)$ for some $\ell_i \geq 0$ by the description of $\mathrm{Pic}(S_i)$. Then $\ell_i \geq \lambda(m_i)$, because $D_i$ has multiplicity at least $m_i$ at $x_i$. However, $\ell_i$ is precisely the multiplicity with which $E_i$ is contained in $D_b \subset X_b \subset X$, so we're done. Note that in this argument, it is irrelevant whether $D_i$ contains $E_i^\infty$ or not. 
\end{proof}

Let $R^\circ$ be the restriction of $R$ to $B^\circ$. Of course, $P^\circ$ is a projective subbundle of $R^\circ$, so we obtain a section over $B^\circ$ of the Grassmannian bundle ${\mathbb G}({\mathrm{rank}{\sh F}}-1,R) \to B$. By properness, this section extends to a section over $B$, so $P^\circ$ can be extended to a projective subbundle $P \subseteq R$. However, since $B$ is affine, $P$ and $R$ can be simultaneously trivialized. Then, it is clear that it is possible to choose sections of $P^\circ \to B^\circ$ as in the paragraph preceding Lemma \ref{lemma inside prop} whose extensions over $B$ pass through any chosen point of $P_b$. Then the lemma implies that
$$ P_b(\kk) \subseteq \left\{ [D_b] \in |{\sh L}_b| = R_b(\kk): \mathrm{coeff}_{E_i} D_b \geq \lambda(m_i) \right\}, $$
hence
$$ \mathrm{rank} {\sh F} - 1 \leq \max \left\{ -1, \dim |{\sh L}_b| - \sum_{i=1}^{n} \lambda(m_i) \right\} $$
and the conclusion of the proposition follows since $\mathrm{rank} {\sh F} = h^0({\sh E}_t)$ and $\dim |{\sh L}_b| = \dim |{\sh L}_t|$.  \end{proof}

\subsection{Some special cases} Below are some applications of Proposition \ref{main theorem}.

\begin{thm}\label{main theorem 2}
Assume that $\mathrm{char}(\kk)=0$. Let $M \in \{1,2,3,...,\infty\}$. 

(a) Let $S$ be the Atiyah ruled surface and ${\sh L}_\ell = {\sh O}_S(F_q+\ell E_\infty)$, with the same notation as above. Then the linear system $|{\sh L}_\ell(m_1,m_2,...,m_n)|$ is nonspecial for all $\ell$ and all $m_1,m_2,...,m_n \leq M$, if and only if this condition holds for $n=1$.

(b) Assume that the condition in part (a) is true (say, for general $E$). Then, if $(S,{\sh L})$ is a general genus $g$ primitively polarized K3 surface, the linear system $|{\sh L}(m_1,m_2,...,m_n)|$ is nonspecial for all $\ell$ and all $m_1,m_2,...,m_n \leq M$. 

(c) Assume that the condition in part (a) is true (say, for general $E$). Let $d$, $m_1,m_2,...,m_n$ be positive integers, $n \geq 10$, such that 
\begin{equation}
\begin{aligned}
m_j=
\begin{cases}
d & \text{ if } j \leq 8, \\
d-1 & \text{ if } j = 9, \\
\leq M & \text{ if } j \geq 10.
\end{cases}
\end{aligned}
\end{equation}
Then the linear system $|{\sh L}_{3d}(m_1,m_2,...,m_n)|$ of degree $3d$ plane curves with $n$ general fat points of the specified multiplicities is nonspecial. 
\end{thm}

Note that although the requirement in part (c) is obviously very restrictive, it may still cover a large number of cases up to Cremona transformations. We lack a clear statement or explanation for this elementary arithmetical assertion.

\begin{proof}
All three parts of the theorem will be proved by applying directly Proposition \ref{main theorem}, so all we need to do is to construct the suitable specializations in all cases. We remark that verifying condition (\ref{hi vanishing}) is straightforward in all cases (and perhaps also well-known). Indeed, the $\mathrm{H}^2$-vanishing is trivial in all cases by Serre duality. The arguments given below required for the verification of the second bullet give as immediate corrolaries the values of $h^0({\sh L}_b)$; the vanishing of $\mathrm{H}^1({\sh L}_b)$ boils down to a purely numerical verification with Riemann-Roch, which is left to the reader.

Saying that the statement in part (a) of the theorem holds for $n=1$ is equivalent to saying that Conjecture \ref{main conjecture} holds in characteristic zero for all $m \leq M$ (actually, for the given elliptic curve, if we're not assuming it to be general).

(a) There exists a family of surfaces $X \to \mathrm{Spec}({\kk}[t]) = {\mathbb A}^1$ such that the restriction $\mathrm{Spec}({\kk}[t]_{(t)}) \times_{{\mathbb A}^1} X$ is the trivial family $S \times \mathrm{Spec}({\kk}[t]_{(t)})$ with fiber $S$, whereas the central fiber $X_{(t)}$ is isomorphic to $E \times {\mathbb P}^1$. Indeed, the extension
\begin{equation}\label{???}
0 \longrightarrow {\sh O}_{E \times {\mathbb A}^1} \longrightarrow {\sh W} \longrightarrow {\sh O}_{E \times {\mathbb A}^1} \longrightarrow 0
\end{equation}
corresponding to $\smash{ t \otimes 1 \in \mathrm{Ext}^1({\sh O}_{E \times {\mathbb A}^1},{\sh O}_{E \times {\mathbb A}^1}) \cong {\kk}[t] \otimes {\kk} }$ seen over the affine line is a family of Atiyah bundles specializing to a trivial rank two bundle and we may construct $\smash{ X = \underline{\mathrm{Proj}}_{E \times {\mathbb A}^1}\mathrm{Sym }({\sh W}^\vee) \stackrel{\tau}{\to} E \times {\mathbb A}^1 }$. Again, let $E \times {\mathbb A}^1 \cong E_\infty \times {\mathbb A}^1 \hookrightarrow X$ be the section corresponding to the first term in the extension above. 

The the divisor class in this case is ${\sh L} = {\sh O}_X(\ell E_\infty \times {\mathbb A}^1 + \tau^{-1}(q \times {\mathbb A}^1))$ and the elliptic fibration of the central fiber is obviously given by projection to the first factor. Note that ${\sh L}|_{X_{(t)}}$ is ${\sh O}_E(q) \boxtimes {\sh O}(\ell)$ on $X_{(t)} \cong E \times {\mathbb P}^1$, so the second bullet in the statement of Proposition \ref{main theorem} is verified. 

Verifying the first bullet requires a more subtle argument. Let $E_u$ be a curve of the form $\{\mathrm{pt}\} \times E \subset X_{(t)}$ other than the respective copy of $E_\infty$. Of course, the short exact sequence for the normal bundles of the inclusions $E_u \subset X_{(t)} \subset X$ shows that ${\sh N}_{E_u/X}$ is indeed an extension of the structure sheaf by itself. 

Assume by way of contradiction that the extension was split. This implies that $E_u$ admits a first order deformation inside $X$ which is flat over ${\mathbb I}_2 := \mathrm{Spec}({\kk}[t]/(t^2))$, which in turn induces a global section $z$ of ${\sh W}^{(2)} = {\sh W}|_{{\mathbb I}_2}$. Restricting (\ref{???}) to ${\mathbb I}_2$, we obtain a cohomology long exact sequence
\begin{equation}\label{????}
0 \longrightarrow \mathrm{H}^0( {\sh O}_{E \times {\mathbb I}_2} ) \longrightarrow  \mathrm{H}^0( {\sh W}^{(2)}) \stackrel{\theta}{\longrightarrow} \mathrm{H}^0( {\sh O}_{E \times {\mathbb I}_2} ) \longrightarrow \mathrm{H}^1( {\sh O}_{E \times {\mathbb I}_2} ).
\end{equation}
By construction, $\theta(z) \notin (t)$, which is the maximal ideal and hence the maximal submodule of $\mathrm{H}^0( {\sh O}_{E \times {\mathbb I}_2} ) \cong {\kk}[t]/(t^2)$, hence $\theta$ is surjective. Therefore, the last map in (\ref{????}) is identically zero, which implies that the restriction of (\ref{???}) to $E \times {\mathbb I}_2$ splits, contradiction.

Thus all requirements of Proposition \ref{main theorem} are satisfied and (\ref{(3.1)}) reads
\begin{equation}\label{repeat}
\dim |{\sh L}_\ell(m_1,m_2,...,m_n)| \leq \max \left\{ -1, \dim |{\sh L}_\ell | - \sum_{i=1}^{n} \lambda(m_i) \right\}.
\end{equation}
By the assumption for $n=1$, we have $\lambda(m_i) = {m_i+1 \choose 2}$ for all $i$, so the right hand side is the expected dimension and we're done.

(b) Consider a family $K \to B$ of smooth genus $g$ K3 surfaces such that the central fiber $K_b$ is a general elliptically fibered K3 surface over ${\mathbb P}^1$ with a $(-2)$-section which we'll denote by $G$. The polarization ${\sh L}_b$ on $K_b$ is ${\sh L}_b = {\sh O}_{K_b}(G+gF)$, where $F$ denotes an elliptic fiber. We refer the reader to ~\cite{[Ch99]} for an algebraic proof of the existence of such degenerations. Let $E$ be an arbitrary smooth fiber, whose intersection point with $G$ is denoted by $q$. First, we check that $|{\sh L}_b|$ satisfies the second bullet in the statement of Proposition \ref{main theorem}. Note that
$$ {\sh L}_b|_E = {\sh O}_{K_b}(G+gF)|_E = {\sh O}_E(q), $$
so all divisors in $|{\sh L}_b|$ have to intersect any elliptic fiber $E$ not contained in their support at the point $q=E \cap G$ because $q$ is not rationally equivalent to any other point on $E$. It follows that any $D \in |{\sh L}_b|$ is a sum of $g$ fibers and the section $G$. Indeed, since $(D \cdot F) = 1$, any irreducible component $C$ of $D$ must satisfy either $(C \cdot F) = 0$, case in which it is an elliptic fiber, or $(C \cdot F) =1$ and $C$ appears with multiplicity $1$ in $D$, when the arguments above imply that $C=G$ .

Second, we check that the first bullet in Proposition \ref{main theorem} is satisfied. The Kodaira-Spencer class of the first order deformation of $K_b$ is determined by the extension
$$ 0 \longrightarrow {\sh T}_{K_b} \longrightarrow {\sh T}_K|_{K_b} \longrightarrow {\sh O}_{K_b} \longrightarrow 0, $$
where the third nonzero term ought to be interpreted as the normal bundle of $K_b$ in $K$. We have a short exact sequence for the normal bundles of $E \subset K_b \subset K$
$$ 0 \longrightarrow {\sh N}_{E/K_b} \longrightarrow {\sh N}_{E/K} \longrightarrow {\sh N}_{K_b/K}|_E \longrightarrow 0. $$
Of course, the first and third nonzero terms are isomorphic to the structure sheaf ${\sh O}_E$, so ${\sh N}_{E/K}$ is an extension of the structure sheaf by itself. If the Kodaira--Spencer class is general in the sense of ~\cite[Remark 2.2]{[Ch99]}, an argument essentially identical to ~\cite[Proposition 2.1]{[Ch99]} shows that the extension above doesn't split. Hence we are indeed in the setup of Proposition \ref{main theorem}. Thanks to the assumption that (a) is true, we have $\lambda^\mathrm{gen}_0(m_i) = {m_i+1 \choose 2}$ for all $i$ and we're done.

(c) Let $\beta:\mathrm{Bl}_{\{p_1,p_2,...,p_9\}}{\mathbb P}^2 \to {\mathbb P}^2$ be the blowup of the projective plane at nine arbitrary points $p_1,p_2,...,p_9$ and $E_1,E_2,...,E_9$ the respective exceptional divisors. Consider the linear system ${\sh L} = \beta^*{\sh O}(3d) \otimes {\sh O}(-dE_1-...-dE_8-(d-1)E_9)$. Of course, the claim is equivalent to saying that $|{\sh L}(m_{10},m_{11},...,m_n)|$ has the expected dimension, if the nine blown up points are general. 

Imitating the idea in ~\cite[\S6]{[BrLe01]}, we allow the $9$ points to specialize to the base locus of a general pencil of cubics. In the special case, the elements of $|{\sh L}|$ are sums of $d$ members of the pencil and $E_9$. Indeed, we may argue similarly to part (b) above: the restriction of ${\sh L}$ to any smooth elliptic fiber $F$ is 
\begin{equation*}
\begin{aligned}
{\sh L}|_F &= \left[\beta^*{\sh O}(3d) \otimes {\sh O}(-dE_1-...-dE_8-dE_9)\right]|_F \otimes {\sh O}_F(E_9 \cap F) \\
&= {\sh O}(dF')|_F \otimes {\sh O}_F(E_9 \cap F) = {\sh O}_F(E_9 \cap F),
\end{aligned}
\end{equation*}
so all divisors in $|{\sh L}|$ are forced to intersect any fiber $F$ not contained in their support at $E_9 \cap F$, justifying the claim. The rest of the argument is also essentially identical to that in part (b) and it is left to the reader. \end{proof}

Corollary \ref{ugly} together with part (b) of the theorem above imply Theorem \ref{clear main theorem} stated in the introduction. 

\subsection{Proof of Theorem \ref{example theorem}} The emptiness in characteristic zero follows from (\ref{repeat}) and Proposition \ref{proposition 3.6}. Indeed, note that (\ref{repeat}) is true independently of the assumption in part (a) of Theorem \ref{main theorem}, which was only used in the proof after (\ref{repeat}) was stated. 

To prove non-emptiness in characteristic $p$, we construct the divisors explicitly using \ref{proposition 2.2} and simply count all degrees and multiplicities. Let $p_1,p_2,...,p_n$ be the $n$ general points inside $S$. Let $\Gamma$ be the disjoint union of the $n$ fat points $\mathrm{Spec}({\sh O}_{S,p_i}/\mathfrak{m}_{p_i}^p)$. For simplicity of notation, let $N = n {p+1 \choose 2}$. The map
$$ \mathrm{H}^0(S,{\sh O}_S(F_q + NE_\infty) ) \longrightarrow \mathrm{H}^0(\Gamma,{\sh O}_S(F_q + NE_\infty) \otimes {\sh O}_\Gamma) $$
has nonzero kernel for obvious dimension reasons, so there exists an effective divisor $D \in |F_q + NE_\infty|$ of multiplicity at least $p$ at each $p_i$. By Proposition \ref{proposition 2.2}, for each $i$, there exists an elliptic curve $E_i \in |pE_\infty|$ passing through $p_i$. The left hand side inequality in (\ref{0.2cond}) amounts to $\ell \geq N + \sum_{i=1}^n p(m_i-p)$, so
$$ D:=D_0 + \sum_{i=1}^n (m_i-p)D'_i + \left[\ell - N - \sum_{i=1}^n p(m_i-p)\right] E_\infty \in |F_q + \ell E_\infty | $$
is an example of an effective divisor in the desired class with multiplicity at least $m_i$ at $p_i$ for all indices $i$, which completes the proof.

\end{document}